\documentclass[journal]{IEEEtran}
\usepackage{cite}
\usepackage{setspace}
\usepackage{graphicx}
\usepackage{amsmath, amsthm, amssymb, bbm}
\usepackage{verbatim}
\usepackage{psfrag}
\usepackage{bm}
\usepackage{color}
\usepackage[tight,footnotesize]{subfigure}
\usepackage{algorithm,algpseudocode,algorithmicx,comment}
\usepackage{microtype}
\newcommand{\bd}[1]{\textbf{#1}}

\newtheorem{theorem}{Theorem}
\newtheorem{proposition}{Proposition}

\title{Pricing Residential Electricity Based on Individual Consumption Behaviors}
\author{\IEEEauthorblockN{
Siddharth Patel\IEEEauthorrefmark{1},
Raffi Sevlian\IEEEauthorrefmark{2},
Baosen Zhang\IEEEauthorrefmark{3}, and
Ram Rajagopal\IEEEauthorrefmark{1} \\
\IEEEauthorblockA{\IEEEauthorrefmark{1}Department of Civil and Environmental Engineering, Stanford University, Stanford, CA, 94305, USA} \\
\IEEEauthorblockA{\IEEEauthorrefmark{2}Department of Electrical Engineering, Stanford University, Stanford, CA, 94305, USA} \\
\IEEEauthorblockA{\IEEEauthorrefmark{3}Department of Electrical Engineering, University of Washington, Seattle, WA, 98195, USA}
\thanks{This research was supported in part by the TomKat Center for Sustainable Energy, the Precourt Institute for Energy Efficiency, and the Thomas V. Jones Stanford Graduate Fellowship in Science \& Engineering.} }}

\begin{document}
%

%
%
%

%
%

\markboth{DRAFT}%
{Shell \MakeLowercase{\textit{et al.}}: Bare Demo of IEEEtran.cls for Journals}
%



\maketitle

\begin{abstract}
The conventional practice of retail electric utilities is to aggregate customers geographically.
The utility purchases electricity for its customers via bulk transactions on the wholesale market, and it passes these costs along to its customers, the end consumers, through their rate plan.
Typically, all residential consumers are offered the same per unit rate plan, which leads to cost sharing.
Some consumers use their electricity at peak hours, when it is more expensive on the wholesale market, and others consume mostly at off peak hours, when it is cheaper, but they all enjoy the same per unit rate through their utility.
This paper proposes a method for the utility to segment a population of consumers on the basis of their individual consumption patterns.
An optimal recruitment algorithm is developed to aggregate consumers into groups with a relatively low per unit cost of electricity on the wholesale market.
Enough consumers are grouped to ensure reduced forecast error and consequently diminish wholesale electricity costs for the aggregator.  
The resulting optimal rate groups are stable in that no one consumer can unilaterally improve her outcome. 
\end{abstract}


%
\IEEEpeerreviewmaketitle

\section{Introduction}
%
%
%
%

\IEEEPARstart{E}{lectrical} retail utilities function as intermediaries between the wholesale market and end consumers. 
Utilities purchase in bulk for the consumers they represent and then factor the cost of these bulk purchases into the rate plans offered to the consumers.
This middleman service provides advantages for the system operator and the consumers.
The system operator can deal with a mass of customers through a single agent instead of undertaking many thousands of transactions with end consumers.
The consumers are spared the complexities of entering the wholesale market.
A more fundamental benefit of aggregating consumers is the reduction of uncertainty in load forecasts. 
While the day-ahead consumption of individual consumers is difficult to forecast to within 50\% accuracy, the aggregate consumption of a large group of consumers can be accurately forecasted with errors smaller than 2\% \cite{Christiaanse71,Hagan87,Sevlian14}.
Thus, by aggregating a group of customers, utilities can provide the system operator with higher-accuracy forecasts of load.
As a consequence, utilities face less risk in their transactions on the wholesale market.

Traditionally, the cost of the bulk electricity is allocated to the residential consumers using simple proportional rate plans: all customers face a single rate, and each customer's bill is computed by multiplying that rate by the customer's total usage.
However, most utilities are moving away from this simple scheme because of \emph{cost sharing} among customers.
The wholesale market price varies greatly based on time of day and geographical location.\footnote{For California, prices in the afternoon at the inland areas tends to be much higher than prices at night for coastal areas.}
If all consumers face a single rate, then those that use electricity at cheaper times or locations are subsidizing those that use electricity at more expensive times or locations.
Enabled by advances in customer metering (e.g. advanced metering infrastructure), utilities have started designing rates more representative of usage.

New rates mainly take two forms: geography-based and time-based. For example, Pacific Gas and Electric Company (PG\&E) divides northern California into zones and charges a different rate for each zone.
Customers also face time of use (TOU) prices which charge different prices for electricity consumed at different times of the day.
However, both methods still do not take individual consumer behaviors into account.
It is known that consumers from the same geographical area may exhibit drastically different behaviors \cite{Patel2014}, as illustrated in Figure \ref{fig:two_users}.
Also, time of use prices may lead to large changes in customer bills since small shifts in temporal behaviors can result in significant cost differences.
Furthermore, prior work in \cite{Ito12} established that residential consumers are sensitive to the average price of electricity, more so than to the marginal price.
Therefore a constant average price is also more beneficial in influencing customer behaviors.
We investigate how an agent can provide differing average prices for its customers. 


\begin{figure}[ht]
\centering
\includegraphics[width=3.5in,trim={0 0 0 15mm},clip]{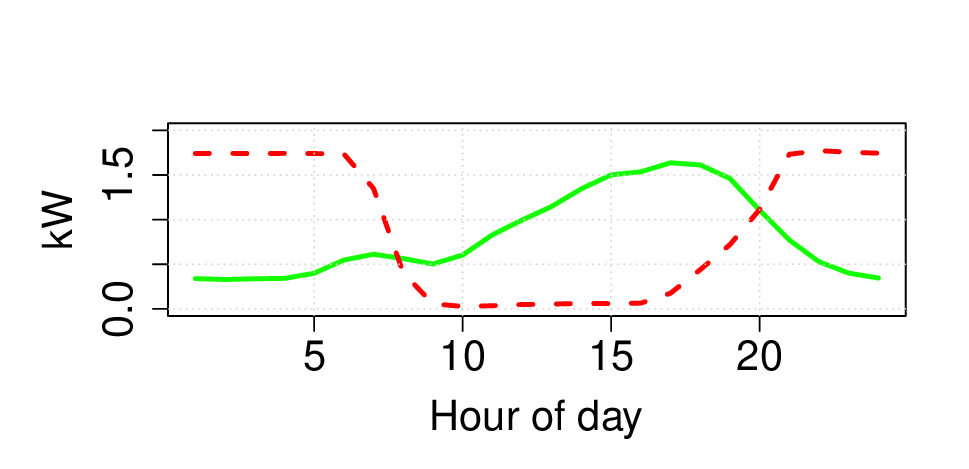}
\\[-4mm]
\caption{The two curves are the mean daily usage patterns for two different consumers in Bakersfield, California, over a period of one year. The solid green curve is for a consumer who uses electricity mostly at times when the price of electricty is high, whereas the dashed red curve is for a consumer who uses electricity mostly at off-peak times. This demonstrates that two consumers nearby each other can have very different consumption behaviors.}
\label{fig:two_users}
\end{figure}

In this paper, we design an optimal aggregation and pricing scheme by explicitly considering the behaviors of individual consumers.
We aggregate consumers into groups where a single rate is charged for all consumers within a group.
This design balances the two competing desires of rate design: limiting cost sharing and reducing uncertainty.
We propose a method to optimally segment consumers into groups for given levels of uncertainty tolerated by the groups.
Our scheme can be used by any agent that purchases electricity for consumers on the wholesale market, such as a utility or other load-serving entity (LSE).
For the remainder of this paper, we will use LSE to denote any such agent.

We model the electricity market as a simple two-stage market. At the day ahead stage, the LSE forecasts its customers' consumption and purchases some amount of electricity \cite{Allcott09,Roozbehani12}.
Any discrepancy between the amount purchased in advance and the customers' actual consumption is settled at the real time price.
Thus, the LSE faces risks in its wholesale market transactions due to forecast errors and to uncertainty in the real time price, whose value is unknown at the day ahead stage and can spike sharply \cite{Kirschen04}.
Aggregation enables the LSE to mitigate the risk of high forecast errors.

An important application of our method is designing rate plans in deregulated retail markets.
The deregulation of the retail electricity market has opened the possibility for LSEs to offer a variety of rate plans for residential consumers.
In ERCOT's geography, for example, there are over 200 plans with a range of prices and schemes.
These plans have experienced significant turnover of their customer base, which is undesirable from the standpoint of an individual LSE \cite{AEE2013,Carson2012}.
We show that our design method induces a stable partitioning of the consumers, in the sense that no consumer can reduce his cost by unilaterally moving from one group to another.

\subsection{Our approach}

We develop a fractional integer program for aggregating groups of consumers with the lowest per unit cost of electricity, and we present an optimal solution to this nonconvex problem.
The LSE then uses the group's average per unit cost as the basis for its rate.
As the group size gets larger, this average per unit cost increases, but the forecast error decreases \cite{Sevlian14}.
We quantify this trade-off in our dataset of hourly smart meter readings for over 100,000 residential consumers for one year.
We propose a method for the LSE to determine its preferred group size and composition.
The LSE can extend this method to segment an entire population of consumers into different groups based on their average cost to serve and a common acceptable level of forecast error.
This segmentation scheme creates groups of consumers that are stable in the sense that no one consumer can improve her situation unilaterally.

\subsection{Outline}

In Section II of this paper, we formalize the aggregation problem facing the LSE.
In Section III, we present an optimal algorithm for constructing groups of consumers with a low per unit cost of electricity.
In Section IV, we calculate the uncertainty faced by a group of consumers and demonstrate the trade-off between rate and uncertainty.
In Section V, we show how to segment the entire population using our method.
We close with some concluding remarks in Section VI.

\section{Problem Statement}

We consider the scenario of an LSE who purchases electricity on behalf of a group of residential consumers in a two-stage wholesale market. We begin by defining the costs this LSE incurs for purchasing electricity.
We introduce a selection vector $\textbf{u}\,\epsilon\,\left\lbrace0,1\right\rbrace^N$, where $N$ is the size of the entire population of consumers, and $\textbf{u}_i = 1$ if the $i$th consumer is part of the group of consumers for which the LSE purchases electricity.
Let $\textbf{d}^{(i)}$ denote the electricity consumption of the $i$th consumer on a given day.
We assume that this consumption is nonnegative (i.e., no reverse power flow).
The total consumption of the group of consumers on the given day is
\begin{equation}
\textbf{d} = \Sigma_{i=1}^N \textbf{u}_i\textbf{d}^{(i)}
\label{eq:d_def}
\end{equation}
 
The two stages of the wholesale electricity market are the day ahead market and the real time market.
At the day ahead stage, the LSE forecasts the aggregate consumption of the consumers for the next day, $\hat{\textbf{d}}$.
Based on that forecast, the LSE purchases an amount of electricity for them, $\tilde{\textbf{d}}$, at the day ahead price, $\textbf{p}$.
The cost incurred at the day ahead stage is $\textbf{p}^T\tilde{\textbf{d}}$.
The next day, the consumers' actual consumption, $\textbf{d}$, is realized. 
The difference between what they consume and what the LSE previously purchased is settled at the real time price, $\textbf{q}$.
We assume that \textbf{p} and \textbf{q} are nonnegative.

For the $k$th day, the total cost $c_k$ that the LSE pays is:
\begin{equation}
c_k = \textbf{p}_k^{T}\tilde{\textbf{d}}_k + \textbf{q}_k^{T}(\textbf{d}_k-\tilde{\textbf{d}}_k).
\label{eqn:Initial_ck}
\end{equation}
Over a period of $K$ days, the LSE will pay an average per unit cost of electricity $r_K$, given by:
\begin{equation}
r_K = \frac{\Sigma_{k=1}^K c_k}{\Sigma_{k=1}^K \textbf{1}^{T}\textbf{d}_k} = \frac{\frac{1}{K}\Sigma_{k=1}^K c_k}{\frac{1}{K}\Sigma_{k=1}^K \textbf{1}^{T}\textbf{d}_k}.
\label{eq:Initial_rk}
\end{equation}
(Note that the rates for regulated utilities are defined in a similar way as the average per unit cost over a period of time.)

The amount of electricity purchased at day-ahead $\tilde{\textbf{d}}_k$ may be different than the forecasted consumption $\hat{\textbf{d}}_k$.
For example, the LSE may regularly decide to purchase slightly more than the forecasted amount in order to avoid the risk of paying a very high real time penalty.
On the other hand, the real time price and the day ahead price are typically nearly equal in expectation\cite{Kirschen04}, so the LSE needs to carefully balance conservatism and cost.

The LSE seeks to minimize $r_K$, which is a random variable.
For mathematical simplicity, instead of working with the expectation of the entire ratio, we focus on the numerator.
Suppose that the LSE enters into a contract with the consumers it is representing for a period of $K$ days.
For large $K$, we can replace the numerator in the right-most fraction in equation (\ref{eq:Initial_rk}) with the expected value of $c_k$:
\begin{equation}
r_K = \frac{K\,\mathbb{E}_{\textbf{d}_k}\left[c_k\right]}{\Sigma_{k=1}^K \textbf{1}^{T}\textbf{d}_k}.
\label{eq:RK_with_expectation}
\end{equation}
We can expand the expectation in the numerator of (\ref{eq:RK_with_expectation}):
\begin{equation}
\mathbb{E}_{\textbf{d}_k}\left[c_k\right] = \mathbb{E}_{\textbf{d}_k}\left[\textbf{p}_k^{T}\tilde{\textbf{d}}_k + \textbf{q}_k^{T}(\textbf{d}_k-\tilde{\textbf{d}}_k)\right].
\label{eq:Exp_ck}
\end{equation}

We can write $\hat{\textbf{d}}_k = \textbf{d}_k + \boldsymbol{\epsilon}_k$; in other words, the forecasted consumption is the actual consumption plus an error term, $\boldsymbol{\epsilon}_k$.
As discussed previously, the LSE can choose at the day ahead stage to buy a different amount of electricity than its forecast.
Let the difference between the two be $\boldsymbol{\delta}_k$.
Then $\tilde{\textbf{d}}_k = \hat{\textbf{d}}_k + \boldsymbol{\delta}_k$.
We now rewrite the right hand side of (\ref{eq:Exp_ck}) as:
\begin{equation}
\mathbb{E}_{\textbf{d}_k}\left[\textbf{p}_k^{T}(\textbf{d}_k + \boldsymbol{\epsilon}_k+\boldsymbol{\delta}_k)\right] + \mathbb{E}_{\textbf{d}_k}\left[\textbf{q}_k^{T}(-\boldsymbol{\epsilon}_k - \boldsymbol{\delta}_k)\right].
\label{eq:Exp_ck_with_eps_delta}
\end{equation}
We apply the fact that $\mathbb{E}\left[xy\right]=\text{Cov}(x,y)+\mathbb{E}\left[x\right]\mathbb{E}\left[y\right]$ to expand (\ref{eq:Exp_ck_with_eps_delta}):
\begin{subequations}
\begin{align*}
& =  \mathbb{E}_{\textbf{d}_k}\left[\textbf{p}_k^T\textbf{d}_k\right] & + \text{tr}\left(\text{Cov}\left(\textbf{p}_k,\boldsymbol{\epsilon}_k\right)\right) +\mathbb{E}_{\textbf{d}_k}\left[\textbf{p}_k\right]^T\mathbb{E}_{\textbf{d}_k}\left[\boldsymbol{\epsilon}_k\right] \\
& & + \text{tr}\left(\text{Cov}\left(\textbf{p}_k,\boldsymbol{\delta}_k\right)\right) +\mathbb{E}_{\textbf{d}_k}\left[\textbf{p}_k\right]^T\mathbb{E}_{\textbf{d}_k}\left[\boldsymbol{\delta}_k\right] \\
& & - \text{tr}\left(\text{Cov}\left(\textbf{q}_k,\boldsymbol{\epsilon}_k\right)\right) -\mathbb{E}_{\textbf{d}_k}\left[\textbf{q}_k\right]^T\mathbb{E}_{\textbf{d}_k}\left[\boldsymbol{\epsilon}_k\right] \\
& & - \text{tr}\left(\text{Cov}\left(\textbf{q}_k,\boldsymbol{\delta}_k\right)\right) -\mathbb{E}_{\textbf{d}_k}\left[\textbf{q}_k\right]^T\mathbb{E}_{\textbf{d}_k}\left[\boldsymbol{\delta}_k\right] 
\end{align*}
\label{eq:Exp_ck_expanded}
\end{subequations}
We assume that the LSE uses an unbiased forecaster, so $\mathbb{E}\left[\boldsymbol{\epsilon}_k\right] = \textbf{0}$.
Furthermore, we assume that the electricity market is an efficient market, meaning that $\mathbb{E}\left[\textbf{p}_k\right]=\mathbb{E}\left[\textbf{q}_k\right]$.
Under these conditions, we can simplify the above sum to:
\begin{equation}
\mathbb{E}_{\textbf{d}_k}\left[\textbf{p}_k^T\textbf{d}_k\right]+
\text{tr}\left(\text{Cov}\left(\textbf{p}_k-\textbf{q}_k,\boldsymbol{\epsilon}_k\right)\right)+\text{tr}\left(\text{Cov}\left(\textbf{p}_k-\textbf{q}_k,\boldsymbol{\delta}_k\right)\right).
\label{eq:Exp_ck_reduced}
\end{equation}

The following two sections examine how the LSE can use this model to evaluate the per unit cost and uncertainty involved in purchasing electricity for its customers.

\section{Cost-based Aggregation}

We begin by outlining some simplifcations that allow us to focus on the cost of serving an individual consumer and groups of consumers.
In this section we will assume that the forecasting errors are independent of the difference between the day ahead and real time price.
This means that the second term of (\ref{eq:Exp_ck_reduced}), $\text{tr}\left(\text{Cov}\left(\textbf{p}_k-\textbf{q}_k,\boldsymbol{\epsilon}_k\right)\right)$, equals zero.
This assumption is standard in the literature \cite{Bitar2012},\cite{Zhao2013}.\footnote{The independence assumption may not hold generally.
For example, suppose the LSE's forecast errors are correlated positively to the forecast errors of all buyers on the market. Suppose on a given day that the LSE's forecast error is negative.
We would then expect that the total amount of electricity purchased in the market at the day ahead stage will be below the total demand.
This would drive up the real time price due to a shortage of generation scheduled at the day ahead stage.}

The first term in (\ref{eq:Exp_ck_reduced}), $\mathbb{E}_{\textbf{d}_k}\left[\textbf{p}_k^T\textbf{d}_k\right]$, represents how aligned the group's consumption is with the day ahead price.
If the group uses most of its electricity during times when the price is high, then that term will be large. The last term in (\ref{eq:Exp_ck_reduced}), $\text{tr}\left(\text{Cov}\left(\textbf{p}_k-\textbf{q}_k,\boldsymbol{\delta}_k\right)\right)$, captures whether, on average, the LSE saves money or loses money by purchasing something different than what the forecaster predicts for the group's consumption.
We will assume that the LSE is unable to make money in this way.
In this case, the LSE's best strategy is to set $\boldsymbol{\delta}_k = \textbf{0}$, and we are left with $\mathbb{E}_{\textbf{d}_k}\left[c_k\right] = \mathbb{E}_{\textbf{d}_k}\left[\textbf{p}_k^T\textbf{d}_k\right]$.
With these simplifications in place, the rate paid by the LSE per unit of electricity is:
\begin{equation}
r_K = \frac{K\,\mathbb{E}_{\textbf{d}_k}\left[\textbf{p}_k^T\textbf{d}_k\right]}{\Sigma_{k=1}^K \textbf{1}^{T}\textbf{d}_k} = \lambda_K
\label{eq:simple_rate}
\end{equation}

When recruiting consumers, the LSE seeks to minimize the rate that it will pay for their electricity.
The LSE does not know how they will consume electricity in the future, but it can use historical data to compute an estimate of the $\lambda_K$ it would pay if it were serving the $i$th consumer:
\begin{equation}
\hat{\lambda}^{(i)} = \frac{\Sigma_{h=1}^H \textbf{p}_h^T\textbf{d}^{(i)}_h}{\Sigma_{h=1}^H \textbf{1}^{T}\textbf{d}^{(i)}_h},
\label{eqn:lambda_def}
\end{equation}
where $H$ is the number of days of historical data available.
We assume that the consumers' history is a good predictor of their future behavior in terms of how aligned their electrical consumption is with the day-ahead price.

To achieve the lowest possible rate, the LSE would choose to recruit and serve the single consumer with the lowest value of $\hat{\lambda}$.
This would be the consumer whose consumption is most \emph{orthogonal} to the day-ahead price vector, given the assumed nonnegativity of consumption and price.
In other words, this consumer uses electricity mostly at off-peak hours.

This may not be a feasible recruitment plan for the LSE for two main reasons.
First, a single consumer likely will not be a viably large basis for the LSE's operation.
Second, the electrical consumption of an individual residential consumer is highly variable day-to-day and subject to forecasting errors of 50\% \cite{Patel2014}.
This may lead to large fluctuations in the day-to-day per unit cost of electricity paid by the LSE.

Therefore, the LSE needs a method to recruit and aggregate multiple consumers into a group that has a low cost to serve.
For a group specified by the selection vector $\textbf{u}$, we define the cost to serve metric $\hat{\lambda}_{\textbf{u}}$ as follows,

\begin{equation}
\hat{\lambda}_{\textbf{u}} = \frac{\Sigma_{h=1}^H \textbf{p}_h^T\textbf{d}_h}{\Sigma_{h=1}^H \textbf{1}^{T}\textbf{d}_h},
\label{eq:lambda_group_def}
\end{equation}

where $\textbf{d}_h$ is the consumption of the group on the $h$th day.

\subsection{Optimal Recruitment Algorithm}

Suppose the LSE wishes to recruit the $M$ consumers who will have the lowest cost to serve among any group of that size.
Let $\textbf{P}_H$ denote the concatenation of $\textbf{p}_h$ vectors for $h=1$ to $H$.
Similarly, $\textbf{D}_H^{(i)}$ is the concatenation of $\textbf{d}_h^{(i)}$ vectors.
Choosing the $M$ best consumers to minimize $\hat{\lambda}_{\textbf{u}}$ is given by the following optimization problem:

\begin{subequations}
\label{eqn:1}
\begin{align}
\underset{\textbf{u}}{\text{minimize}} \; & \frac{\sum_{i=1}^N \textbf{u}_i \bd P_H^T \bd D_H^{(i)}}{\sum_{i=1}^N \textbf{u}_i \bd 1^T \bd D_H^{(i)}} \\
\mbox{subject to} \; & \bd 1^T \bd u = M,\; \bd u_i \in \{0,1\}.
\end{align}
\end{subequations}

Let $\textbf{t}_i = \bd P_H^T \bd D_H^{(i)}$ and $\textbf{w}_i = \bd 1^T \bd D_H^{(i)}$.  We rewrite the optimization as:
\begin{subequations} \label{eqn:2}
\begin{align}
\underset{\textbf{u}}{\text{minimize}} \; & \frac{\bd u^T \bd t}{\bd u^T \bd w} \\
\mbox{subject to} \; & \bd 1^T \bd u = M,\; \bd u_i \in \{0,1\}.
\end{align}
\end{subequations}
Finally, we introduce a slack variable $\lambda$ and obtain:
\begin{subequations} \label{eqn:3}
\begin{align}
\underset{\textbf{u}, \lambda}{\text{minimize}} \; & \lambda \\
\mbox{subject to} \;& (\bd t - \lambda \bd w)^T \bd u \leq 0\\
& \bd 1^T \bd u = M,\; \bd u_i \in \{0,1\}.
\end{align}
\end{subequations}

In general, the above combinatorial optimization problems, called linear fractional programs, are difficult to solve. Because all parameters in our problem are positive, there is an efficient bisection algorithm for a feasibility problem which tries to find $\bd u$ and $\lambda$ to satisfy (\ref{eqn:3}b) and (\ref{eqn:3}c).
This can be performed in a greedy fashion per Algorithm \ref{algorithm-optimal}.
For a given $\lambda$, we rank each element of the vector $(\bd t - \lambda \bd w)$ and choose the smallest $M$ elements with the selection vector $\bd u$.
If $(\bd t - \lambda \bd w)^T \bd u \leq 0$, then we have found a feasible solution for the given value of $\lambda$; otherwise, no solution exists for this value of $\lambda$. 

\begin{algorithm}
\begin{algorithmic}[1]
\State $\textbf{Initialize bisection method:}$
\State $\overline{\lambda} \leftarrow \max\{t_1/w_1, \hdots, t_N/w_N\}$
\State $\underline{\lambda} \leftarrow \min\{t_1/w_1, \hdots, t_N/w_N\}$
\State $\text{Set}\;\gamma$
\Comment{Convergence threshold}
\State $\textbf{Bisection method:}$
\While{$\overline{\lambda} - \underline{\lambda} > \gamma $}
\State $\lambda = (\underline{\lambda}  + \overline{\lambda})/2$
\Comment{Update current $\lambda$} 
\State \text{Compute}\;$\left( \mathbf{t} - \lambda  \mathbf{w} \right)$
\State $\text{Sort}\left( \mathbf{t} - \lambda  \mathbf{w} \right) \text{in ascending order} \text{ to obtain}  \{i_1 \hdots i_N\} $
\State \text{Construct} $\textbf{u}_{\lambda, i} = 1 \text{ if } i \in \{i_1,  \hdots, i_M\}$
\If{$\left( \mathbf{t} - \lambda  \mathbf{w} \right)^{T} \bd u_{\lambda}  \le 0$}
\Comment{Feasibility test}
\State $\textbf{u}_{\lambda} \text{ feasible}$
\Else
\State $\textbf{u}_{\lambda} \leftarrow \emptyset$
\EndIf
\If{$\bd u_{\lambda} = \emptyset$}
\State $\underline{\lambda}  \leftarrow \lambda$
\Comment {Infeasible, increase lower bound}
\Else
\State $\overline{\lambda}  \leftarrow \lambda$
\Comment {Feasible, decrease upper bound}
\EndIf
\EndWhile
\State $\textbf{Result:}$
\State Minimum $\lambda$
\State Associated selection vector $\textbf{u}_{\lambda}$
\Comment{$\bd 1^T \textbf{u}_{\lambda} = M$}
\end{algorithmic}
\caption{The LSE can use this alorithm to select the group of $M$ customers who had the lowest cost to serve $\hat{\lambda}_{\textbf{u}}$ over a period of historical data.}
\label{algorithm-optimal}
\end{algorithm}

\begin{theorem}
Algorithm \ref{algorithm-optimal} returns the minimum feasible value of $\lambda$ (within a tolerance of $\gamma$).
\end{theorem}

\begin{proof}
\label{Proof of Optimality}

To show the bisection method applied to the feasibility problem is optimal, we must show that a unique minimum $\lambda$ exists and that bisection will always find that $\lambda$.
\newpage
First define the set of feasible values of $\lambda$ as $\Lambda = \{ \lambda : \exists\ \mathbf{u} \in \{0,1\}^N,  \mathbf{1}^T\mathbf{u} = M,  \left( \mathbf{t} - \lambda \mathbf{w} \right)^T\mathbf{u} \le 0 \}$.  This defines a set of attainable servicing costs for M consumers.  Therefore, the solution to (\ref{eqn:3}) is $\lambda^\star = \inf \{\lambda \in \Lambda \}$.  Bisection will find $\lambda^\star$ if the following condition holds for $\epsilon>0$: $\lambda \in \Lambda \implies \lambda + \epsilon \in \Lambda$ and $\lambda \notin \Lambda \implies \lambda - \epsilon \notin \Lambda$.  That is, if there is a unique transition point, then the bisection method is guaranteed to find it.

We first prove that $\lambda \in \Lambda \implies \lambda + \epsilon\in \Lambda $. If $ \lambda  \in \Lambda$, then there exists $\mathbf{u}_{\lambda}$ satisfying $\left( \mathbf{t} - \lambda \mathbf{w} \right)^T\mathbf{u}_\lambda \le 0$.  The vector $\mathbf{u}_{\lambda}$ is feasible for $\lambda + \epsilon$ since $( \mathbf{t} - (\lambda + \epsilon)\mathbf{w} )^T\mathbf{u}_\lambda = ( \mathbf{t} - \lambda\mathbf{w} )^T\mathbf{u}_\lambda  - \epsilon\mathbf{w}^T\mathbf{u}_\lambda \le 0$.  The first term is nonpositive because $\mathbf{u}_{\lambda}$ is feasible for $\lambda$.  The second term is always positive since $\mathbf{w}_i > 0$ and $\epsilon > 0$.

Now we prove that $\lambda \notin \Lambda \implies \lambda - \epsilon \notin \Lambda $. If $\lambda \notin \Lambda$ then $\nexists$ $\mathbf{u}_\lambda$ such that  $\left( \mathbf{t} - \lambda \mathbf{w} \right)^T\mathbf{u}_\lambda \le 0$.  However, $\exists$ $\mathbf{u}^{\prime}$ such that $\left( \mathbf{t} - \lambda \mathbf{w} \right)^T\mathbf{u}^{\prime} \le \left( \mathbf{t} - \lambda \mathbf{w} \right)^T\mathbf{u}$, $ \forall \mathbf{u} : \mathbf{u} \in \{0, 1\}^N$ and $\bd 1^T \bd u = M$.  That is, $\mathbf{u}^{\prime}$ is a selection vector which produces the smallest value of $\left( \mathbf{t} - \lambda \mathbf{w} \right)^T\mathbf{u}$.  We call this vector $\mathbf{u}^{\prime}$ a candidate for $\lambda$.  If $\lambda$ is infeasible, it must be the case that the candidate satisfies $\left( \mathbf{t} - \lambda \mathbf{w} \right)^T\mathbf{u}^{\prime} > 0$; otherwise, $\lambda$ would be feasible. Now consider a candidate for $\lambda - \epsilon$, and call it $\mathbf{u}^{\prime\prime}$.  For $\mathbf{u}^{\prime\prime}$ we can state the following:
\begin{align}
(\mathbf{t} - (\lambda - \epsilon)\mathbf{w})^T\mathbf{u}^{\prime\prime} &=  (\mathbf{t} - \lambda \mathbf{w})^T\mathbf{u}^{\prime\prime} + \epsilon \mathbf{w}^T\mathbf{u}^{\prime\prime} \label{eq-a-1}  \\
&\ge  (\mathbf{t} - \lambda \mathbf{w})^T\mathbf{u}^{\prime} + \epsilon \mathbf{w}^T\mathbf{u}^{\prime\prime} \label{eq-a-2}\\
& > 0
\end{align}
The inequality in \eqref{eq-a-1} - \eqref{eq-a-2} holds since $\mathbf{u}^{\prime}$ is a candidate for $\lambda$ and must therefore  satisfy  $(\mathbf{t} - \lambda \mathbf{w})^T\mathbf{u}^{\prime} \le (\mathbf{t} - \lambda \mathbf{w})^T\mathbf{u}$ for all $\mathbf{u}$, including $\mathbf{u}^{\prime\prime}$.  Since $(\mathbf{t} - (\lambda - \epsilon)\mathbf{w})^T\mathbf{u}^{\prime\prime} > 0$, we conclude that $\lambda-\epsilon$ is infeasible.
\end{proof}

Now the LSE has a method for putting together a group of consumers of size $M$ who will be relatively cheap to service, assuming all of their electricity can be purchased at the day-ahead price.
The LSE must decide how to set the appropriate size $M$.
A larger group will have a higher cost to serve because it will have to include more consumers whose consumption aligns more closely with peaks in electricity prices.
On the other hand, a larger group offers advantages of lower forecasting errors and easier administration.

\section{Aggregating to Mitigate Uncertainty}

In order to analyze the effects of uncertainty on the total cost faced by the LSE, we consider a somewhat different market design.
Suppose the LSE is unable to sell back surplus electricity at the real time settlement.
Such an arrangement is typically captured in deviation charges and is used by the system operator to discourage LSEs from purchasing excessive amounts of electricity on the day ahead market \cite{Bitar2012}.
Alternatively, the operator may want LSEs to commit to specific load profiles at the day ahead stage with a large penalty for consuming less than they committed to.
For this market design, the total cost paid by the LSE in (\ref{eqn:Initial_ck}) becomes:
\begin{equation}
c_k = \textbf{p}_k^{T}\tilde{\textbf{d}}_k + \textbf{q}_k^{T}\left.\left[\textbf{d}_k-\tilde{\textbf{d}}_k\right]\right._+,
\label{eqn:ck_with_tildas}
\end{equation}
where $\left.\left[x\right]\right._+$ denotes $\text{max}(x,0)$.

%
%

For a given group of consumers, the LSE chooses $\tilde{\textbf{d}}_k$ to minimize $\mathbb{E}_{\textbf{d}_k}\left[c_k\right]$. At the optimal choice $\tilde{\textbf{d}}_k^\star$, the following first order condition must hold:
\begin{equation}
\label{eqn:optimal_d_star_first_order}
\nabla_{\tilde{\textbf{d}}_k} \mathbb{E}_{\textbf{d}_k}\left[c_k\right] \bigg |_{\tilde{\textbf{d}}_k=\tilde{\textbf{d}}_k^\star} =\textbf{0}
\end{equation}
The day ahead price is known to the LSE when it purchases electricity on the day ahead market, so $\textbf{p}_k$ is a realized random variable at this stage.
The gradient can be applied as shown:
\begin{subequations}
\label{eqn:optimal_d_star_first_steps}
\begin{align}
\textbf{0} & = \textbf{p}_k + \mathbb{E}_{\textbf{d}_k}\left[\nabla_{\tilde{\textbf{d}}_k}\textbf{q}_k^{T}\left.\left[\textbf{d}_k-\tilde{\textbf{d}}_k^\star\right]\right._+\right]\\
& = \textbf{p}_k + \mathbb{E}_{\textbf{d}_k}\left[-\textbf{q}_k^T\text{diag}(\vec{\textbf{1}}\lbrace\textbf{d}_k>\tilde{\textbf{d}}_k^\star\rbrace)\right],
\end{align}
\end{subequations}
where $\vec{\textbf{1}}\lbrace\textbf{d}_k>\tilde{\textbf{d}}_k^\star\rbrace$ is a vector whose $i$th element is one if the $i$th element of $\textbf{d}_k$ is greater than that of $\tilde{\textbf{d}}_k^\star$.
Next, we assume that the difference between $\textbf{d}_k$ and $\tilde{\textbf{d}}_k^\star$ is independent of the real time price.
\begin{subequations}
\label{eqn:optimal_d_star_last_steps}
\begin{align}
\textbf{0} & = \textbf{p}_k - \mathbb{E}_{\textbf{d}_k}\left[\textbf{q}_k\right]^T \text{diag}( \mathbb{E}_{\textbf{d}_k}\left[\vec{\textbf{1}}\lbrace\textbf{d}_k>\tilde{\textbf{d}}_k^\star\rbrace\right])\\
& = \textbf{p}_k - \mathbb{E}_{\textbf{d}_k}\left[\textbf{q}_k\right]^T \text{diag}(\vec{\mathbb{P}}\lbrace\textbf{d}_k>\tilde{\textbf{d}}_k^\star\rbrace),
\end{align}
\end{subequations}
where $\vec{\mathbb{P}}\lbrace\textbf{d}_k>\tilde{\textbf{d}}_k^\star\rbrace)$ is a vector whose $i$th element is the probability that the $i$th element of $\textbf{d}_k$ is greater than that of $\tilde{\textbf{d}}_k^\star$.
Equation (\ref{eqn:optimal_d_star_last_steps}b) gives an easy way to compute the optimal $\tilde{\textbf{d}}_k^\star$, assuming the distribution of $\textbf{d}_k$ is known.
This is a variation on the classic newsvendor model \cite{Petruzzi1999},\cite{Khouja1999}.

Recall that $\tilde{\textbf{d}}_k^\star=\textbf{d}_k+\boldsymbol{\epsilon}_k+\boldsymbol{\delta}_k^\star$.
Therefore, we can write $\vec{\mathbb{P}}\lbrace\textbf{d}_k>\tilde{\textbf{d}}_k^\star\rbrace$ as $\vec{\mathbb{P}}\lbrace-\boldsymbol{\epsilon}_k>\boldsymbol{\delta}_k^\star\rbrace$, so the optimal day ahead purchase quantities are based on the distribution of the forecast errors.
If we assume that the forecast errors are jointly Gaussian with zero mean and covariance $\boldsymbol{\Sigma}$, we can write $\boldsymbol{\delta}_k^\star = \textbf{f}_k(\boldsymbol{\Sigma})$.
Substituting the optimal electricity purchase into (\ref{eqn:ck_with_tildas}), and taking the expectation, we arrive at:
\begin{subequations}
\label{eqn:exp_ck_with_solution_substituted}
\begin{align}
& \mathbb{E}_{\textbf{d}_k}\left[c_k\right] = \mathbb{E}_{\textbf{d}_k}\left[\textbf{p}_k^T\textbf{d}_k\right] + g_k(\boldsymbol{\Sigma}) \\
& g_k(\boldsymbol{\Sigma}) = \mathbb{E}_{\textbf{d}_k}\left[\textbf{f}_k(\boldsymbol{\Sigma})\right] + \mathbb{E}_{\textbf{d}_k}\left[\textbf{q}_k\right]^T\mathbb{E}_{\textbf{d}_k}\left[\left.\left[-\boldsymbol{\epsilon}_k-\textbf{f}_k(\boldsymbol{\Sigma})\right]\right._+\right].
\end{align}
\end{subequations}

Recall that (\ref{eqn:optimal_d_star_first_order}) relates to a component of (\ref{eq:RK_with_expectation}), so we can substitute (\ref{eqn:exp_ck_with_solution_substituted}a) back into (\ref{eq:RK_with_expectation}) to obtain:
\begin{subequations}
\label{eqn:rk_for_positive_penalty}
\begin{align}
r_K & = \frac{K\,\mathbb{E}_{\textbf{d}_k}\left[\textbf{p}_k^T\textbf{d}_k\right]}{\textbf{1}^{T}\textbf{D}_K} + \frac{\Sigma_{k=1}^K \, g_k(\boldsymbol{\Sigma})}{\textbf{1}^{T}\textbf{D}_K} \\
& = \lambda_K + \frac{\Sigma_{k=1}^K \, g_k(\boldsymbol{\Sigma})}{\textbf{1}^{T}\textbf{D}_K}.
\end{align}
\end{subequations}
In the next section of this paper, we will show how the first and second terms of (\ref{eqn:rk_for_positive_penalty}) vary with the composition and size of the group of consumers selected by the LSE, in other words how they vary with $\textbf{u}$ and $\left. \|\textbf{u}\| \right._1$.
For now we will discuss these relationships in general terms.

The LSE varies $\textbf{u}$ by controlling whom it recruits into its group.
If the LSE recruits only consumers who use most of their electricty at peak hours, its per unit cost of electricity will be higher, and this is captured in $\lambda_K$.
Furthermore, if the LSE recruits only a very small number of consumers, it will face larger forecast errors, which may lead to higher variation in $r_K$, partially through the $g_k(\boldsymbol{\Sigma})$ term. See \cite{Zhang2014} for a treatment of the effect of forecast uncertainty on prices.

Thus, the LSE faces a tradeoff when assembling a group of consumers to service.
After showing some initial empirical relationships in the data, we will propose a heuristic algorithm for the LSE to segment the population into stable groups that can be offered different rates based on how they consume electricity.

\section{Segmenting the population}

\subsection{Data description}

The data are hourly smart meter readings over a period of one year (summer 2010 to summer 2011) for 110,000 residential customers of PG\&E.
In addition, we use day ahead and real time prices published by the California Independent System Operator (CAISO) over the same period.
One limitation of our data is that it only spans one year.
We use part of the data - the first nine months - to evaluate and segment the consumer population, and the rest of the data to validate the segmentation scheme.
Therefore, we may be introducing a seasonal bias into the segmentation in our simulation, but that would readily be fixed by using data spanning an entire year for the segmentation algorithm presented below.

\subsection{Forecasting method}

For the purposes of this paper, we use a relatively simple forecaster.
We use an ARMA model, with temperature as an external regressor, to predict the daily total electricity usage by the consumers, $\hat{y}$.
We use a vector ARMA model, again using temperature as an external regressor, to predict the normalized load shape, $\hat{\textbf{s}}$.
This is a vector whose elements are the fraction of the daily total electricity consumed in each hour.
In other words, $\textbf{1}^T\hat{\textbf{s}} = 1$.
Finally, we multiply the predicted daily total by the predicted normalized load shape to obtain the predicted consumption for the next day: $\hat{\textbf{d}} = \hat{y}\hat{\textbf{s}}$.

\subsection{Range of $\hat{\lambda}^{(i)}$}

To begin with, we evaluate the range of $\hat{\lambda}^{(i)}$ values in the entire population.
Recall that the LSE seeks to minimize the per unit cost of electricity that it incurs on the wholesale market.
One natural way for it to do this is to choose consumers who use electricity at off-peak times with respect to the day ahead price, and $\hat{\lambda}^{(i)}$ captures this alignment for the historical consumption of the $i$th consumer.

\begin{figure}[ht]
\centering
\includegraphics[width=3.5in,trim={0 0 0 15mm},clip]{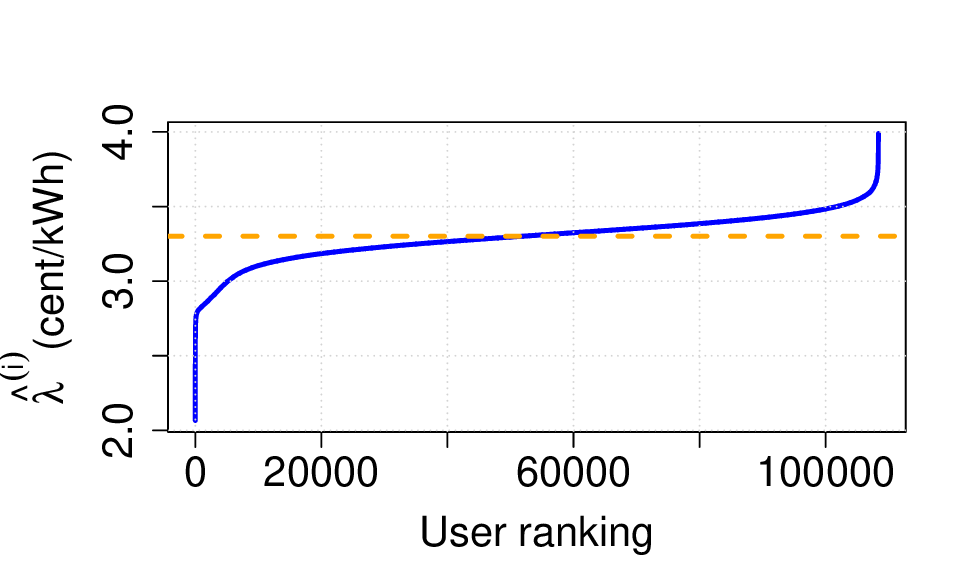}
\\[-4mm]
\caption{A plot of $\hat{\lambda}^{(i)}$ vs. consumer rank, in order of increasingly expensive consumers. The difference between the cheapest and the costliest consumer is about 2.0 cent/kWh. The steep slopes at either end of the curve mean that there are relatively few very cheap consumers or very expensive consumers. The horizontal orange line is the average per unit cost for the population.}
\label{fig:ranked_lambdas}
\end{figure}

Figure \ref{fig:ranked_lambdas} shows how $\hat{\lambda}^{(i)}$ increases as we go from the cheapest consumer to the most expensive consumer in our dataset, ranked by their $\hat{\lambda}^{(i)}$ metric computed over the first nine months of the year.
The most expensive consumer is almost twice as expensive as the cheapest consumer, a significant difference.
This suggests that the LSE may be able to group together cheaper consumers and offer them a lower rate.

\subsection{Application of optimal recruitment algorithm}

As mentioned before, the LSE would likely need to aggregate consumers together to service them practically.
We apply Algorithm \ref{algorithm-optimal} using the first nine months of data as the historical period.
Let $\hat{\lambda}_M$ denote the minimum value of $\lambda$ achieved using the algorithm given a group size $M$.
Figure \ref{fig:lambda_M_increasing} illustrates how $\hat{\lambda}_M$ varies with $M$.
Note that at $M=39,800$, approximately 37\% of the population, $\hat{\lambda}_M$ is within 5\% of the population average.
On the other hand, $\hat{\lambda}_{500}$ is 2.75 cents/kWh, which is 17\% lower than the population average, so the optimal group of size $500$ could be offered a meaningfully cheaper rate plan.

\begin{figure}[ht]
\centering
\includegraphics[width=3.5in,trim={0 0 0 15mm},clip]{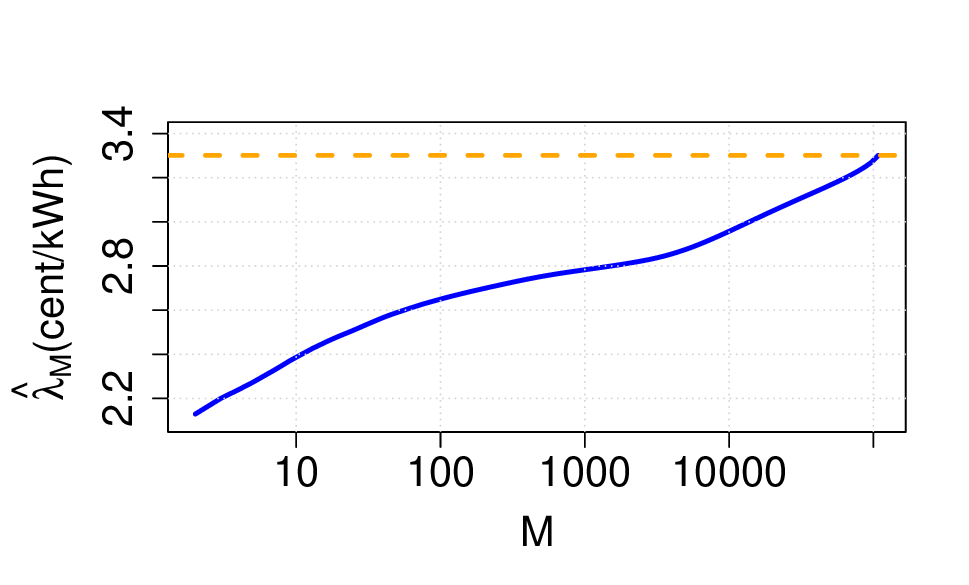}
\\[-4mm]
\caption{$\hat{\lambda}_M$ is increasing in $M$ because as the group size increases, the algorithm includes consumers who are more and more aligned with the day ahead price vector. The horizontal dashed line is $\hat{\lambda}_M$ when $M$ is the size of the entire population. In other words, it is the average per unit cost of electricity for servicing the entire population at the day ahead price. Randomly selected groups of consumers have a $\hat{\lambda}$ close to this population average.}
\label{fig:lambda_M_increasing}
\end{figure}

The LSE now knows how its per unit cost of electricity will change based on how large of a group of low-cost consumers it recruits.
Practical operational considerations may set a lower limit on group size, but the LSE needs a way to determine how big of a group it should aggregate, in other words, how to select $M$.
We propose that the LSE take into account how group size affects its forecasting error.

The larger the group, the more accurately its aggregate consumption can be forecasted.
To demonstrate this result in the data, we use the forecaster described previously to predict hourly group consumption for the last three months of the year.
We do this for two types of groups - groups constructed by randomly assembling individuals, and groups constructed using the optimal recruitment algorithm.
The error metric we use is the coefficient of variation, $CV$, which is commonly used in forecasting literature as a measure of performance.
Let $d(t)$ be the actual consumption of the group of consumers at time $t$, and let $\hat{d}(t)$ be the forecasted consumption for the same period.
Then
\begin{equation}
CV=100\frac{\sqrt{\frac{1}{T}\Sigma_{t=1}^T\big(d(t)-\hat{d}(t)\big)^2}}{\frac{1}{T}\Sigma_{t=1}^Td(t)} (\%).
\label{eqn:CV_def}
\end{equation}

Figure \ref{fig:CV_group_size} illustrates the effect of group size on forecast error.
For both the random groups and optimal groups, the coefficient of variation of the forecast error decreases with group size.
The decrease for random groups is smooth and monotone, while the curve for the optimal groups exhibits more complicated behavior.
If the error curve for optimal groups coincided with that of the random groups, we could assert that $CV$, and by extension $\boldsymbol{\Sigma}$, is a function solely of group size - but that is not the case.
It is clearly a function of both $\textbf{u}$ and $\left. \|\textbf{u}\| \right._1$, both composition and size.

The $CV$ is related to the variability in electricity costs faced by the LSE \cite{Zhang2014}.
If the $CV$ is high, and if the day ahead and real time prices tend to be different, the LSE can expect a higher variation in its per unit cost of electricity from one day to the next because the higher forecast errors will require it to purchase (or sell) more electricity at the real time settlement.

\begin{figure}[ht]
\centering
\includegraphics[width=3.5in,trim={0 0 0 15mm},clip]{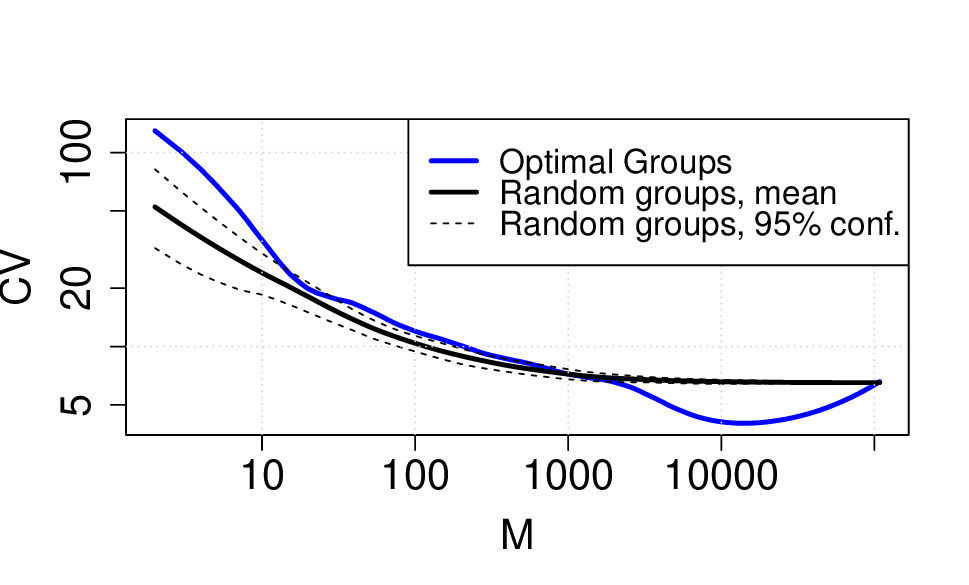}
\\[-4mm]
\caption{$CV$ is decreasing in $M$ because the consumption of larger groups can be forecasted more accurately. The black line is the mean CV for randomly constructed groups, and the dashed lines are the 95\% confidence interval. The blue line is for groups constructed using the optimal recruitment algorithm. Until about $M=1000$, the optimal groups have higher forecasting errors than the random groups.}
\label{fig:CV_group_size}
\end{figure}

Therefore, if the LSE wants to control the variation in its per unit cost of electricity, it must aggregate a certain number of consumers together.
The LSE faces a trade-off - a smaller group size allows for a lower per unit cost, while a larger group size allows for a lower variation in per unit cost.
Suppose the LSE sets an upper limit $\overline{CV}$ on the forecast error as a way of controlling the variation in its costs.
Then it should determine the minimum group size $\underline{M}$ such that the $CV$ for optimal groups of that size is less than or equal to $\overline{CV}$.
Finally, it should run the optimal recruitment algorithm to choose the lowest-cost group of consumers of size $\underline{M}$.
Now the LSE knows which consumers to recruit to get a group that will have a low per unit cost of electricity and an adequately small variation in that cost.
For example, if the LSE set a limit of $\overline{CV}=10\%$, it would find on Figure \ref{fig:CV_group_size} that it needs to recruit a group of size $\underline{M}=209$.
It would then run Algorithm 1 with $M=209$ and obtain the selection vector $\textbf{u}$ that would identify the consumers it should recruit to the group.
The LSE would then refer to Figure \ref{fig:lambda_M_increasing} to determine that the long run average per unit cost for this group will be 2.70 cents per kWh, and it can offer a rate plan to this group based on that cost.

\subsection{Segmenting the entire population}

We now turn to how the LSE could use the methods we've presented to serve an entire population.
The LSE is interested in designing \emph{stable} rate schemes, in which customers do not have an incentive to jump frequently from one to another.
This consideration is relevant to current practice - load serving entities in ERCOT have reported substantial customer turnover, which is undesirable \cite{AEE2013,Carson2012}.
In our method, the LSE will segment the population into groups that have different long run rates but similar levels of variation in daily cost.

The LSE starts by considering the entire population and establishing an accepted level of variation, $\overline{CV}$.
It generates a forecast error curve similar to that for the optimal groups in Figure \ref{fig:CV_group_size}, and it selects the smallest group size $\underline{M}_1$ such that the $CV$ for that group is less than or equal to $\overline{CV}$.
There is an optimal group of low-cost consumers $\textbf{u}_1$ corresponding to this group size $\underline{M}_1$.
The LSE takes this group of consumers and serves it as an aggregate, offering them a rate plan based on their group cost $\hat{\lambda}_{\underline{M}_1}$.

After removing that first set of consumers from the population, the LSE repeats this process on the remaining consumers, with the same $\overline{CV}$ limit, until the entire population is segmented into groups.
Say that the process results in $P$ total groups.
These groups, $\textbf{u}_1, \textbf{u}_2,...,\textbf{u}_P$, will be of various sizes.
Each subsequent group will be offered a higher rate than the preceding group because the lowest-cost consumers have already been assigned to prior groups.
In other words, $\hat{\lambda}_{\underline{M}_i}\leq\hat{\lambda}_{\underline{M}_{i+1}}$.
However, by construction, every group will have the same level of risk, $\overline{CV}$.

Returning to our previous example, if the LSE sets $\overline{CV}=10$, the first group of consumers it recruits will be of size $\underline{M}_1=209$, and it will offer that group a rate plan based on an average per unit cost of $\hat{\lambda}_{\underline{M}_1}=2.70$ cents per kWh.
The LSE then removes this group of consumers from the population.
Figure \ref{fig:Second_Round_CV10} illustrates the next iteration of the LSE's segmentation process.
The LSE again applies the $\overline{CV}=10$ threshold and chooses to go with a group of size $\underline{M}_2=145$, as seen in the top panel.
The LSE recruits these consumers into a group and offers them a rate plan based on its average per unit cost of $\hat{\lambda}_{\underline{M}_2}=2.81$ cents per kWh, as illustrated in the bottom panel.
The LSE would repeat this process until all consumers had been assigned into a group.
Note that it is possible that on the last iteration, the $CV$ curve for the remaining consumers will be entirely above $\overline{CV}$.
In that case, the LSE must aggregate together all of the remaining consumers, and this last group will have a forecast error higher than the other groups.
Alternatively, the LSE can decide not to serve the remaining consumers, or to segment them on the basis of a new, higher $\overline{CV}$ threshold.

\begin{proposition}
This method produces stable pricing plans in the following sense: assuming that all consumers agree to the variation limit $\overline{CV}$, and given the initial segmentation of the population by the LSE, no single consumer can unilaterally take an action that would improve her outcome.
\end{proposition}
\begin{proof}
Suppose that a consumer has been assigned to group $\textbf{u}_i$ and offered that group's average rate, $\hat{\lambda}_{\underline{M}_i}$.
The consumer has no incentive to join the group $\textbf{u}_{i+1}$ because that group pays a higher rate.
The consumer would like to join the group $\textbf{u}_{i-1}$ to enjoy their lower rate.
However, the existing consumers in group $\textbf{u}_{i-1}$ will not agree to this because they already have enough consumers in their group to reduce their risk level to $\overline{CV}$.
Adding this new consumer will only lead to an increase in their rate because we know that the optimal selection algorithm did not choose this consumer to be part of $\textbf{u}_{i-1}$.
\end{proof}

\begin{figure}[ht]
\centering
\subfigure{\includegraphics[width=3.5in,trim={0 0 0 15mm},clip]{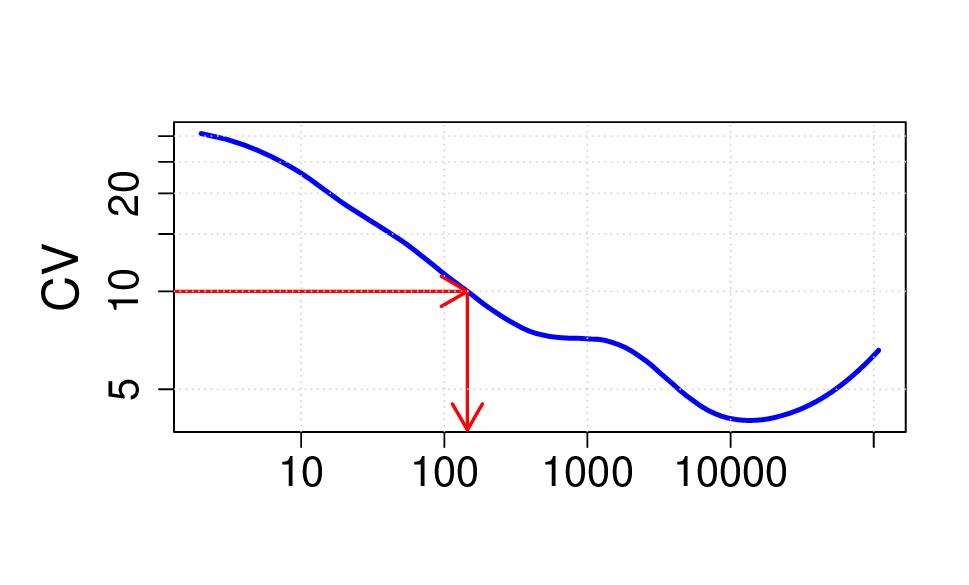}\label{fig:Second_CVs}}
\\[-10mm]
\subfigure{\includegraphics[width=3.5in,trim={0 0 0 15mm},clip]{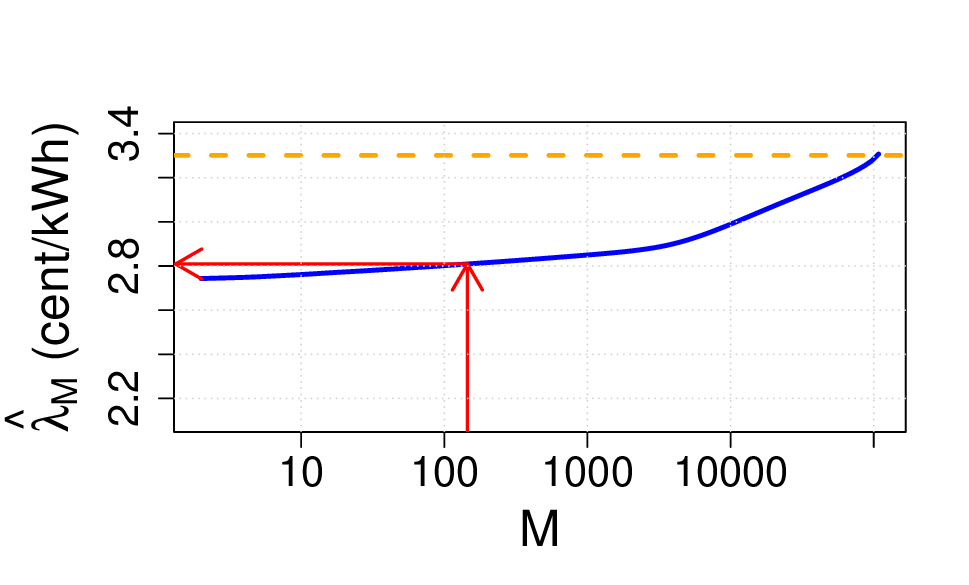}\label{fig:Second_Lambdas}}
\\[-4mm]
\caption{These two figures illustrate the next step in the segmentation process when using $\overline{CV}=10$ as the limit. The LSE faces a new curve for $CV$ vs. group size $M$. The LSE applies the $\overline{CV}=10$ threshold to the forecast error curve, determines the required group size, and then reads off their group per unit cost from the bottom figure. The red arrows trace out that process. Note that the starting point for the optimal group $\hat{\lambda}_M$ is noticeably higher than it was for the entire population in Fig. \ref{fig:lambda_M_increasing}. This makes sense given that the cheapest consumers were already recruited into the first group, so the remaining consumers will have higher per unit costs. The orange dashed line on the $\hat{\lambda}_M$ plot is the per unit cost for servicing the entire population, the same value as in Fig. \ref{fig:lambda_M_increasing}.}
\label{fig:Second_Round_CV10}
\end{figure}

\begin{figure}[ht]
\centering
\subfigure{\includegraphics[width=3.5in,trim={0 0 0 15mm},clip]{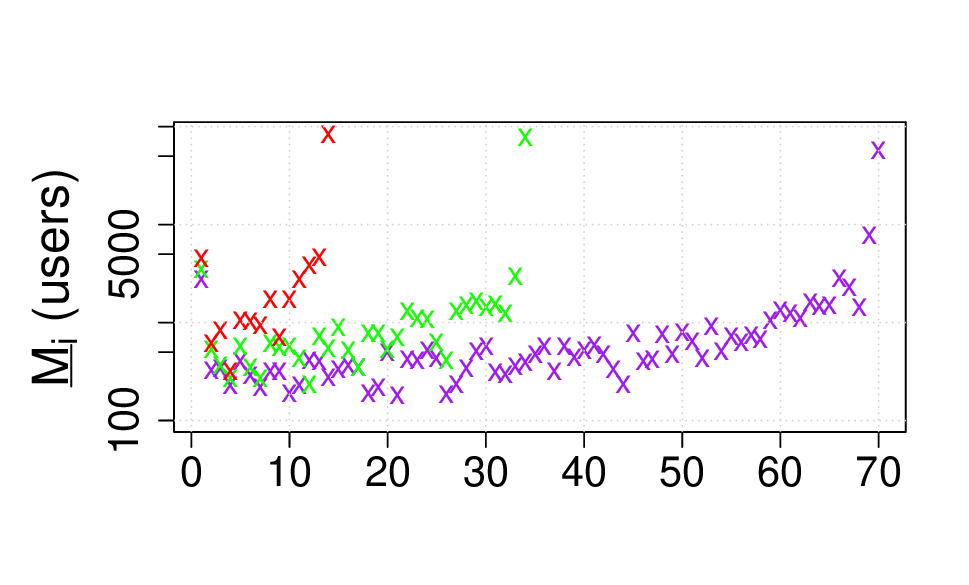}\label{fig:Segmented_Ms}}
\\[-10mm]
\subfigure{\includegraphics[width=3.5in,trim={0 0 0 15mm},clip]{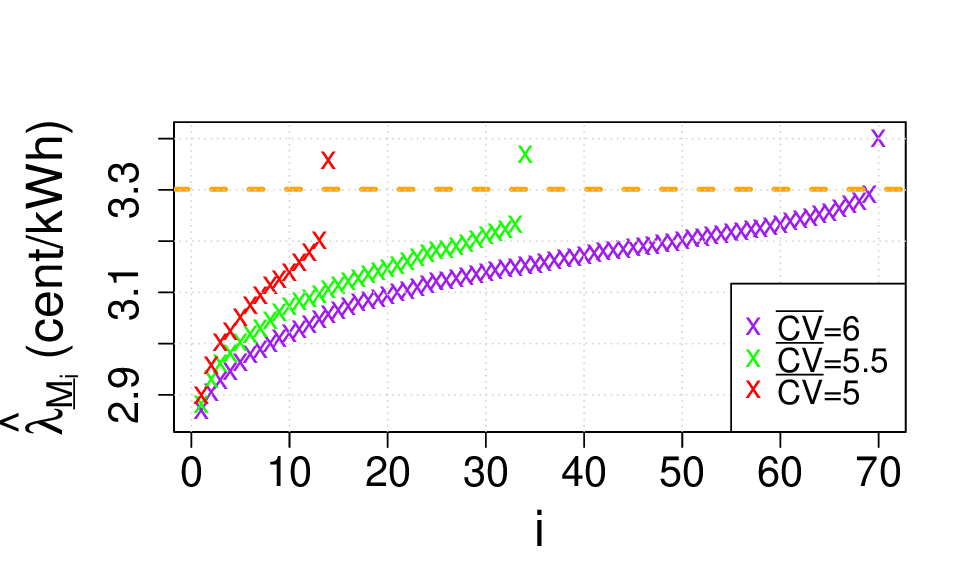}\label{fig:Segmented_Lambdas}}
\\[-4mm]
\caption{These figures illustrate the results of segmenting the entire population of consumers using three different values of the $\overline{CV}$ threshold.
The first figure shows how the optimal group size $\underline{M}_i$ varies as the segmentation proceeds.
The $\overline{CV}=6$ segmentation proceeds with smaller group sizes because it doesn't require aggregating as many consumers together to achieve that forecast error threshold.
In each of the three cases, the last and largest group is well over 50,000 people. These last groups did not achieve the required $\overline{CV}$ threshold.
The second figure shows the per unit costs for servicing each of the groups.
The orange dashed line is the per unit cost for the entire population, a reference point for this segmentation.
As expected, in each case, the first groups have per unit costs below the population average, and the later groups approach and then exceed that average.}
\label{fig:Segmentations}
\end{figure}

\begin{figure}[ht]
\centering
\includegraphics[width=3.5in,trim={0 0 0 15mm},clip]{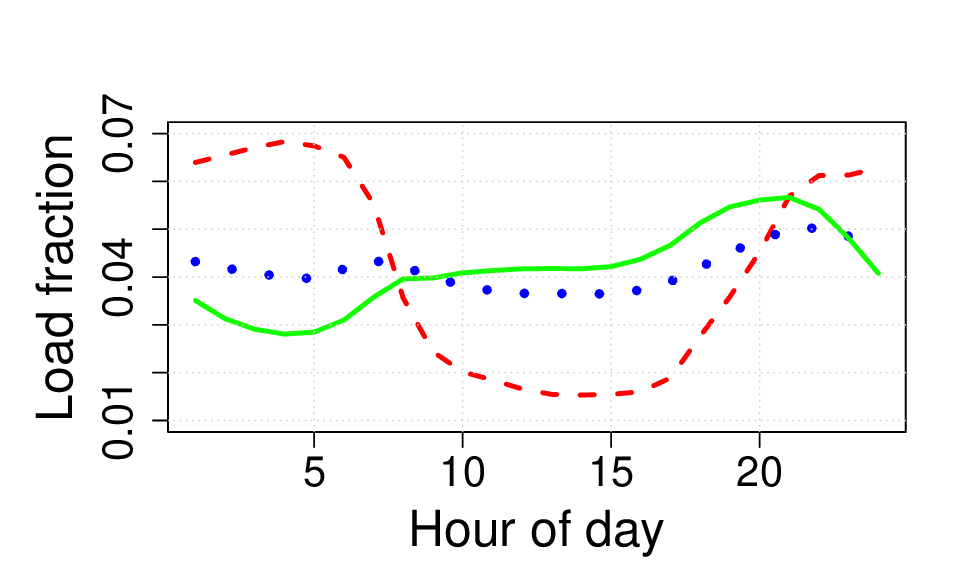}
\\[-4mm]
\caption{The three curves are the average daily aggregate load shapes for three groups created by the segmentation algorithm using $\overline{CV}=6$. The dashed red curve is for the second group segmented out. Its consumption is mostly off peak, and its per unit cost to serve is 2.90 cents/kWh. The dotted blue curve is for the 22nd group segmented out, with per unit cost to serve of 3.10 cents/kWh. The solid green curve is for the 64th group segmented out. Its consumption is mostly on peak, and its per unit cost to serve is 3.25 cents/kWh.}
\label{fig:three_groups}
\vspace{-2mm}
\end{figure}

We simulate this pricing design scheme on the entire population of consumers in our dataset, using $\overline{CV} = 6$, $5.5$, and $5$.
We evaluate the daily per unit cost of electricity that the LSE would incur for servicing each group.
Figure \ref{fig:Segmentations} illustrates the results of our simulation.
As expected, when using the higher forecast error threshold, the LSE segments the population into smaller groups.
For all three cases, the groups segmented out earlier incur costs noticeably lower than the social average cost, whereas the groups formed later incur higher costs.

For $\overline{CV}=6$, the LSE segments the population into 70 groups. The first 69 all satisfy the forecast error threshold and have per unit costs ranging from $\hat{\lambda}_{\underline{M}_1}=2.87$ to $\hat{\lambda}_{\underline{M}_{69}}=3.29$ cents per kWh.
The average size of these groups is 738 consumers.
All of these groups, covering almost 51,000 consumers, have a per unit cost lower than the population average.
The 70th group lumps together the remaining 57,000 consumers, with a per unit cost of $\hat{\lambda}_{\underline{M}_{70}}=3.40$, which is above the population average.
This last group has a forecast error of $CV=7.9$, exceeding the $\overline{CV}=6$ threshold.

For comparison, when using $\overline{CV}=5$, the first 13 groups satisfy the forecast error threshold and have costs ranging from 2.90 to 3.20 cents per kWh.
The average size of these groups is 1892 consumers, and in total, they include about 25,000 consumers.
The remaining 83,000 consumers are placed into the last group, with a per unit cost of 3.36 cents per kWh, which is above the population average, and a forecast error of $CV=7.1$, well above the threshold.

Figure \ref{fig:three_groups} shows the average daily aggregate load shapes for three groups ($i=2,22,64$) from the segmentation performed with $\overline{CV}=6$. The groups have distinct consumption patterns, reflected in their differing per unit costs. This demonstrates the effectiveness of the segmentation algorithm.

\section{Conclusion}

We developed and proposed a method for an LSE to segment a population of electricity consumers on the basis of their average per unit cost.
Using their historical consumption patterns, the LSE identifies and aggregates consumers who are cheaper to serve.
The LSE aggregates enough consumers to reduce the forecast error to an acceptable level, which is common across all consumer groups.
We quantified the trade-offs involved in this process when segmenting a population of over 100,000 PG\&E residential consumers.
Our simulation demonstrated that the LSE can offer each group a different average per unit rate plan that is based on the extent to which they consume their electricity at peak times.

\section*{Acknowledgment}
We would like to thank Pacific Gas and Electric Company for providing us with the smart meter data used in this study.

\ifCLASSOPTIONcaptionsoff
  \newpage
\fi




\bibliographystyle{IEEEtran}
\bibliography{citation}

\end{document}